\documentclass[reqno,12pt]{amsart}%
\usepackage{amssymb}
\usepackage{amsfonts}
\usepackage{amsmath}
\usepackage{graphicx}
\usepackage{amscd,amsthm}%
\newtheorem{theorem}{Theorem}
\theoremstyle{plain}
\newtheorem{acknowledgement}{Acknowledgement}

\newtheorem{corollary}{Corollary}

\newtheorem{definition}{Definition}

\newtheorem{remark}{Remark}

\numberwithin{equation}{section}

\begin{document}
\author{}
\title{}
\maketitle

\begin{center}
\thispagestyle{empty} \pagestyle{myheadings} \markboth{\bf Yilmaz Simsek
}{\bf Functional equations from generating functions:...}

\textbf{{\Large Functional equations from generating functions: a novel
approach to deriving identities for the \textbf{Bernstein basis functions}}}

\bigskip

\textbf{Yilmaz Simsek}\\[0pt]

Department of Mathematics, Faculty of Science University of Akdeniz TR-07058
Antalya, Turkey, ysimsek@akdeniz.edu.tr\newline

\bigskip

\textbf{{\large {Abstract}}}\medskip
\end{center}

\begin{quotation}
The main aim of this paper is to provide a novel approach to deriving
identities for the Bernstein polynomials using functional equations. We derive
various functional equations and differential equations using generating
functions. Applying these equations, we give new proofs for some standard
identities for the Bernstein basis functions, including formulas for sums,
alternating sums, recursion, subdivision, degree raising, differentiation and
a formula for the monomials in terms of the Bernstein basis functions. We also
derive many new identities for the Bernstein basis functions based on this
approach. Moreover, by applying the Laplace transform to the generating
functions for the Bernstein basis functions, we obtain some interesting series
representations for the Bernstein basis functions.
\end{quotation}

\bigskip


\bigskip


\section{Introduction}

The Bernstein polynomials have many applications: in approximations of
functions, in statistics, in numerical analysis, in $p$-adic analysis and in
the solution of differential equations. It is also well-known that in Computer
Aided Geometric Design polynomials are often expressed in terms of the
Bernstein basis functions. These polynomials are called \textit{Bezier curves
and surfaces}.

Many of the known identities for the Bernstein basis functions are currently
derived in an \textit{ad hoc} fashion, using either the binomial theorem, the
binomial distribution, tricky algebraic manipulations or blossoming. The main
purpose of this work is to construct novel functional equations for the
Bernstein polynomials. Using these functional equations and Laplace transform,
we develop a novel approach both to standard and to new identities for the
Bernstein polynomials. Thus these polynomial identities are just the residue
of a much more powerful system of functional equations.

The remainder of this study is organized as follows: We find several
\textit{functional equations} and \textit{differential equations} for the
Bernstein basis functions using generating functions. From these equations,
many properties of the Bernstein basis functions are then derived. For
instance, we give a new proof of the recursive definition of the Bernstein
basis functions as well as a novel derivation for the two term formula for the
derivatives of the $n$th degree Bernstein basis functions. Using functional
equations, we give new derivations for the sum and alternating sum of the the
Bernstein basis functions and a formula for the monomials in terms of the
Bernstein basis functions. We also derive identities corresponding to the
degree elevation and subdivision formulas for Bezier curves. We prove many new
identities for the Bernstein basis functions. Finally, we give some
applications of the Laplace transform to the generating functions for the
Bernstein basis functions. We obtain interesting series representations for
the Bernstein basis functions. We also give some remarks and observations
related to the Fourier transform and complex generating functions for the
Bernstein basis functions.

\section{Generating Functions}

The Bernstein polynomials and related polynomials\ have been studied and
defined in many different ways, for examples by $q$-series, complex functions,
$p$-adic Volkenborn integrals and many algorithms. In this section, we provide
novel generating functions for the Bernstein basis functions.

The Bernstein basis functions $B_{k}^{n}(x)$ are defined as follows:

\begin{definition}%
\begin{equation}
B_{k}^{n}(x)=\left(
\begin{array}
[c]{c}%
n\\
k
\end{array}
\right)  x^{k}(1-x)^{n-k}, \label{1BB3a}%
\end{equation}
where%
\[
\left(
\begin{array}
[c]{c}%
n\\
k
\end{array}
\right)  =\frac{n!}{k!(n-k)!},
\]
$k=0,1,\cdots,n$ cf. \cite{AcikgozSerkan}-\cite{Simsek Acikgoz}.
\end{definition}

Generating functions for the Bernstein basis functions can be defined as follows:

\begin{definition}
\label{Def-1}%
\begin{equation}
f_{\mathbb{B},k}(x,t)=\sum_{n=0}^{\infty}B_{k}^{n}(x)\frac{t^{n}}{n!}.
\label{A-0}%
\end{equation}

\end{definition}

Note that there is one generating function for each value of $k$.

\begin{theorem}
\label{TT1}%
\begin{align}
f_{\mathbb{B},k}(x,t)  &  =\frac{t^{k}x^{k}e^{(1-x)t}}{k!}.\label{1BB3}\\
& \nonumber
\end{align}

\end{theorem}

\begin{proof}
By substituting (\ref{1BB3a}) into the right hand side of (\ref{A-0}), we get%
\[
\sum_{n=0}^{\infty}B_{k}^{n}(x)\frac{t^{n}}{n!}=\sum_{n=0}^{\infty}\left(
\left(
\begin{array}
[c]{c}%
n\\
k
\end{array}
\right)  x^{k}(1-x)^{n-k}\right)  \frac{t^{n}}{n!}.
\]
Therefore%
\[
\sum_{n=0}^{\infty}B_{k}^{n}(x)\frac{t^{n}}{n!}=\frac{\left(  xt\right)  ^{k}%
}{k!}\sum_{n=k}^{\infty}(1-x)^{n-k}\frac{t^{n-k}}{\left(  n-k\right)  !}.
\]
The right hand side of the above equation is a Taylor series for $e^{(1-x)t}$,
thus we arrive at the desired result.
\end{proof}

We give some alternative forms of the generating functions in (\ref{A-0}) as
follows:%
\begin{equation}
\sum_{n=0}^{\infty}B_{k}^{n}(x)\frac{t^{n}}{n!}e^{xt}=\frac{t^{k}x^{k}e^{t}%
}{k!}, \label{g1}%
\end{equation}%
\begin{equation}
\sum_{n=0}^{\infty}B_{k}^{n}(x)\frac{t^{n}}{n!}e^{-t}=\frac{t^{k}x^{k}e^{-xt}%
}{k!}, \label{g2}%
\end{equation}
and%
\begin{equation}
\sum_{n=0}^{\infty}B_{k}^{n}(x)\frac{t^{n}}{n!}e^{(x-1)t}=\frac{t^{k}x^{k}%
}{k!}. \label{g3}%
\end{equation}
By using the above alternative forms we derive some new identities for the
Bernstein basis functions.

\begin{remark}
If we replace $x$ by $\frac{x-a}{b-a}$in (\ref{1BB3}), where $a<b$, then%
\[
\frac{t^{k}\left(  \frac{x-a}{b-a}\right)  ^{k}e^{(\frac{b-x}{b-a})t}}%
{k!}=\sum_{n=0}^{\infty}B_{k}^{n}(x,a,b)\frac{t^{n}}{n!},
\]
where $B_{k}^{n}(x,a,b)$ denotes the generalized Bernstein basis function
defined by:%
\[
B_{k}^{n}(x,a,b)=\left(
\begin{array}
[c]{c}%
n\\
k
\end{array}
\right)  \frac{\left(  x-a\right)  ^{k}(b-x)^{n-k}}{(b-a)^{m}}%
\]
cf. \cite{GoldmanBOOK}.
\end{remark}

A \textit{Bernstein polynomial} $\mathcal{P}(x)$ is a polynomial represented
in the Bernstein basis functions:%
\begin{equation}
\mathcal{P}(x)=\sum_{k=0}^{n}c_{k}^{n}B_{k}^{n}(x) \label{1BB2}%
\end{equation}
cf. \cite{GoldmanBOOK}. Simsek \cite{SimsekSpringer}-\cite{SimsekAMC}, Simsek
\textit{et al}. \cite{Simsek Acikgoz}\ and Acikgoz \textit{et al}.
\cite{AcikgozSerkan} also studied on the generating function for Bernstein
basis function.

\section{Identities for the Bernstein basis functions}

In this section, we use the generating functions for the Bernstein basis
functions to derive a family of functional equations. Using these equations,
we derive a collection of identities for the Bernstein basis functions.

\subsection{Sums and Alternating sums}

From (\ref{1BB3}), we get the following functional equations:%
\begin{equation}
\sum_{k=0}^{\infty}f_{\mathbb{B},k}(x,t)=e^{t} \label{A-1}%
\end{equation}
and%
\begin{equation}
\sum_{k=0}^{\infty}(-1)^{k}f_{\mathbb{B},k}(x,t)=e^{(1-2x)t}. \label{A-2}%
\end{equation}

\begin{theorem}
\label{T2} (Sum of the Bernstein basis functions)%
\[
\sum_{k=0}^{n}B_{k}^{n}(x)=1.
\]

\end{theorem}

\begin{proof}
From (\ref{A-1}), one finds that%
\begin{equation}
\sum_{k=0}^{\infty}f_{\mathbb{B},k}(x,t)=\sum_{n=0}^{\infty}\frac{t^{n}}{n!}.
\label{A-3}%
\end{equation}
By combining (\ref{A-0}) and (\ref{A-3}), we get%
\[
\sum_{n=0}^{\infty}\left(  \sum_{k=0}^{n}B_{k}^{n}(x)\right)  \frac{t^{n}}%
{n!}=\sum_{n=0}^{\infty}\frac{t^{n}}{n!}.
\]
Comparing the coefficients of $\frac{t^{n}}{n!}$ on the both sides of the
above equation, we arrive at the desired result.
\end{proof}

\begin{theorem}
\label{T3} (Alternating sum of the Bernstein basis functions)%
\[
\sum_{k=0}^{n}(-1)^{k}B_{k}^{n}(x)=(1-2x)^{n}.
\]

\end{theorem}

\begin{proof}
By combining (\ref{A-2}) and (\ref{A-3}), we obtain%
\[
\sum_{n=0}^{\infty}\left(  \sum_{k=0}^{n}(-1)^{k}B_{k}^{n}(x)\right)
\frac{t^{n}}{n!}=\sum_{n=0}^{\infty}\frac{(1-2x)^{n}t^{n}}{n!}.
\]
Comparing the coefficients of $\frac{t^{n}}{n!}$ on the both sides of the
above equation, we arrive at the desired result.
\end{proof}

\begin{remark}
Goldman \cite{GoldmanBook3}-\cite[Chapter 5, pages 299-306]{GoldmanBook2}
derived the formula for the alternating sum of the Bernstein basis functions algebraically.
\end{remark}

\subsection{Subdivision}

From (\ref{1BB3}), we have the following functional equation:%
\begin{equation}
f_{\mathbb{B},j}(xy,t)=f_{\mathbb{B},j}\left(  x,ty\right)  e^{t\left(
1-y\right)  }. \label{1BB1Sub}%
\end{equation}
From this functional equation, we get the following identity which is the
basis for subdivision of Bezier curves cf. ( \cite{GoldmanBOOK},
\cite{GoldmanBook2}, \cite{GoldmanBook3}, \cite{SimsekSpringer}).

\begin{theorem}
\label{Theorem3}%
\[
B_{j}^{n}(xy)=%
%TCIMACRO{\dsum \limits_{k=j}^{n}}%
%BeginExpansion
{\displaystyle\sum\limits_{k=j}^{n}}
%EndExpansion
B_{j}^{k}(x)B_{k}^{n}(y).
\]

\end{theorem}

\begin{proof}
By equations (\ref{1BB3}) and (\ref{1BB1Sub})%
\[
\sum_{n=j}^{\infty}B_{j}^{n}(xy)\frac{t^{n}}{n!}=\left(  \sum_{n=0}^{\infty
}B_{j}^{n}(x)y^{n}\frac{t^{n}}{n!}\right)  \left(  \sum_{n=0}^{\infty}%
\frac{\left(  1-y\right)  ^{n}t^{n}}{n!}\right)  .
\]
Therefore%
\[
\sum_{n=j}^{\infty}B_{j}^{n}(xy)\frac{t^{n}}{n!}=\sum_{n=j}^{\infty}\left(
%TCIMACRO{\dsum \limits_{k=j}^{n}}%
%BeginExpansion
{\displaystyle\sum\limits_{k=j}^{n}}
%EndExpansion
B_{j}^{k}(x)\frac{y^{k}\left(  1-y\right)  ^{n-k}}{k!\left(  n-k\right)
!}\right)  t^{n}.
\]
Substituting (\ref{1BB3a}) into the above equation, we arrive at the desired result.
\end{proof}

\begin{remark}
Theorem \ref{Theorem3} is a bit tricky to prove with algebraic manipulations.
Goldman \cite{GoldmanBook3}-\cite[Chapter 5, pages 299-306]{GoldmanBook2}
proved this identity algebraically. He also proved the following related
subdivision identities:%
\[
B_{j}^{n}(\left(  1-y\right)  x+y)=%
%TCIMACRO{\dsum \limits_{k=0}^{j}}%
%BeginExpansion
{\displaystyle\sum\limits_{k=0}^{j}}
%EndExpansion
B_{j-k}^{n-k}(x)B_{k}^{n}(y),
\]
and%
\[
B_{j}^{n}(\left(  1-y\right)  x+yz)=%
%TCIMACRO{\dsum \limits_{k=0}^{n}}%
%BeginExpansion
{\displaystyle\sum\limits_{k=0}^{n}}
%EndExpansion
\left(
%TCIMACRO{\dsum \limits_{p+q=j}}%
%BeginExpansion
{\displaystyle\sum\limits_{p+q=j}}
%EndExpansion
B_{p}^{n-k}(x)B_{q}^{k}(z)\right)  B_{k}^{n}(y).
\]
For additional identities, see cf. \cite{GoldmanBook3}-\cite[Chapter 5, pages
299-306]{GoldmanBook2}.
\end{remark}

\subsection{Formula for the monomials in terms of the Bernstein basis
functions}

Multiplying both sides of (\ref{1BB3}) by $\left(
\begin{array}
[c]{c}%
k\\
l
\end{array}
\right)  $, we get%
\[
\left(
\begin{array}
[c]{c}%
k\\
l
\end{array}
\right)  \frac{\left(  xt\right)  ^{k}}{k!}e^{t(1-x)}=\left(
\begin{array}
[c]{c}%
k\\
l
\end{array}
\right)  \sum_{n=0}^{\infty}B_{k}^{n}(x)\frac{t^{n}}{n!}.
\]
Summing both sides of the above equation over $k$, we obtain the following
functional equation, which is used to derive a formula for the monomials in
terms of the Bernstein basis functions:%
\begin{equation}
\frac{x^{l}t^{l}}{l!}e^{t}=\sum_{k=0}^{\infty}\left(
\begin{array}
[c]{c}%
k\\
l
\end{array}
\right)  f_{\mathbb{B},k}(x,t). \label{A-4}%
\end{equation}

\begin{theorem}%
\[
\left(
\begin{array}
[c]{c}%
n\\
l
\end{array}
\right)  x^{l}=%
%TCIMACRO{\dsum \limits_{l=0}^{k}}%
%BeginExpansion
{\displaystyle\sum\limits_{l=0}^{k}}
%EndExpansion
\left(
\begin{array}
[c]{c}%
k\\
l
\end{array}
\right)  B_{k}^{n}(x)
\]

\end{theorem}

\begin{proof}
Combining (\ref{A-0}) and (\ref{A-4}), we get%
\[
\frac{x^{l}}{l!}\sum_{n=0}^{\infty}\frac{t^{n+l}}{n!}=\sum_{n=0}^{\infty
}\left(  \sum_{k=0}^{n}\left(
\begin{array}
[c]{c}%
k\\
l
\end{array}
\right)  B_{k}^{n}(x)\right)  \frac{t^{n}}{n!}.
\]

Therefore%
\[
\sum_{n=0}^{\infty}\left(  \left(
\begin{array}
[c]{c}%
n\\
l
\end{array}
\right)  x^{l}\right)  \frac{t^{n}}{n!}=\sum_{n=0}^{\infty}\left(  \sum
_{k=0}^{n}\left(
\begin{array}
[c]{c}%
k\\
l
\end{array}
\right)  B_{k}^{n}(x)\right)  \frac{t^{n}}{n!}.
\]
Comparing the coefficients of $\frac{t^{n}}{n!}$ on the both sides of the
above equation, we arrive at the desired result.
\end{proof}

\subsection{Differentiating the Bernstein basis functions}

In this section we give higher order derivatives of the Bernstein basis
functions. We begin by observing that%
\begin{equation}
f_{\mathbb{B},k}(x,t)=g_{k}(t,x)h(t,x), \label{n1}%
\end{equation}
where%
\[
g_{k}(t,x)=\frac{t^{k}x^{k}}{k!}%
\]
and%
\[
h(t,x)=e^{(1-x)t}.
\]

Using Leibnitz's formula for the $l$th derivative, with respect to $x$, we
obtain the following \textit{higher order partial differential equation}:%
\begin{equation}
\frac{\partial^{l}f_{\mathbb{B},k}(x,t)}{\partial x^{l}}=%
%TCIMACRO{\dsum \limits_{j=0}^{l}}%
%BeginExpansion
{\displaystyle\sum\limits_{j=0}^{l}}
%EndExpansion
\left(
\begin{array}
[c]{c}%
l\\
j
\end{array}
\right)  \left(  \frac{\partial^{j}g_{k}(t,x)}{\partial x^{j}}\right)  \left(
\frac{\partial^{l-j}h(t,x)}{\partial x^{l-j}}\right)  . \label{n2}%
\end{equation}
From this equation, we arrive at the following theorem:

\begin{theorem}
\label{TeoX}%
\begin{equation}
\frac{\partial^{l}f_{\mathbb{B},k}(x,t)}{\partial x^{l}}=%
%TCIMACRO{\dsum \limits_{j=0}^{l}}%
%BeginExpansion
{\displaystyle\sum\limits_{j=0}^{l}}
%EndExpansion
\left(
\begin{array}
[c]{c}%
l\\
j
\end{array}
\right)  (-1)^{l-j}t^{l}f_{\mathbb{B},k-j}(x,t). \label{A-6}%
\end{equation}

\end{theorem}

\begin{proof}
Formula (\ref{A-6}) follows immediately from (\ref{n2}).
\end{proof}

Applying Theorem \ref{TeoX}, we obtain a new derivation for the higher order
derivatives of the Bernstein basis functions.

\begin{theorem}
\label{TeoX1}%
\begin{equation}
\frac{d^{l}B_{k}^{n}(x)}{dx^{l}}=\frac{n!}{(n-l)!}%
%TCIMACRO{\dsum \limits_{j=0}^{l}}%
%BeginExpansion
{\displaystyle\sum\limits_{j=0}^{l}}
%EndExpansion
(-1)^{l-j}\left(
\begin{array}
[c]{c}%
l\\
j
\end{array}
\right)  B_{k-j}^{n-l}(x). \label{A7}%
\end{equation}

\end{theorem}

\begin{proof}
By substituting the right hand side of (\ref{A-0}) into (\ref{A-6}), we get%
\[
\sum_{n=0}^{\infty}\left(  \frac{d^{l}B_{k}^{n}(x)}{dx^{l}}\right)
\frac{t^{n}}{n!}=\sum_{n=0}^{\infty}\left(
%TCIMACRO{\dsum \limits_{j=0}^{l}}%
%BeginExpansion
{\displaystyle\sum\limits_{j=0}^{l}}
%EndExpansion
(-1)^{l-j}\left(
\begin{array}
[c]{c}%
l\\
j
\end{array}
\right)  B_{k-j}^{n}(x)\right)  \frac{t^{n+l}}{n!}.
\]
Therefore%
\[
\sum_{n=0}^{\infty}\left(  \frac{d^{l}B_{k}^{n}(x)}{dx^{l}}\right)
\frac{t^{n}}{n!}=\sum_{n=0}^{\infty}\left(
%TCIMACRO{\dsum \limits_{j=0}^{l}}%
%BeginExpansion
{\displaystyle\sum\limits_{j=0}^{l}}
%EndExpansion
(-1)^{l-j}\left(
\begin{array}
[c]{c}%
l\\
j
\end{array}
\right)  \left(
\begin{array}
[c]{c}%
n\\
l
\end{array}
\right)  l!B_{k-j}^{n-l}(x)\right)  \frac{t^{n}}{n!}%
\]
Comparing the coefficients of $\frac{t^{n}}{n!}$ on the both sides of the
above equation, we arrive at the desired result.
\end{proof}

Substituting $l=1$ into (\ref{A7}), we arrive at the following \ standard corollary:

\begin{corollary}
\label{BeT-2}%
\[
\frac{d}{dx}B_{k}^{n}(x)=n\left(  B_{k-1}^{n-1}(x)-B_{k}^{n-1}(x)\right)  .
\]
cf. \cite{AcikgozSerkan}-\cite{Simsek Acikgoz}.
\end{corollary}

\subsection{\textbf{Recurrence Relation}}

In the previous section we computed the derivative of (\ref{n1}) with respect
to $x$ to derive a derivative formula for the Bernstein basis functions. In
this section we are going to differentiate (\ref{n1}) with respect to $t$ to
derive a recurrence relation for the Bernstein basis functions.

Using Leibnitz's formula for the $v$th derivative, with respect to $t$, we
obtain the following \textit{higher order partial differential equation}:%
\begin{equation}
\frac{\partial^{v}f_{\mathbb{B},k}(x,t)}{\partial t^{v}}=%
%TCIMACRO{\dsum \limits_{j=0}^{v}}%
%BeginExpansion
{\displaystyle\sum\limits_{j=0}^{v}}
%EndExpansion
\left(
\begin{array}
[c]{c}%
v\\
j
\end{array}
\right)  \left(  \frac{\partial^{j}g_{k}(t,x)}{\partial t^{j}}\right)  \left(
\frac{\partial^{v-j}h(t,x)}{\partial t^{v-j}}\right)  . \label{n3}%
\end{equation}
From the above equation, we have the following theorem:

\begin{theorem}
\label{TheoremRECUr}%
\begin{equation}
\frac{\partial^{v}f_{\mathbb{B},k}(x,t)}{\partial t^{v}}=%
%TCIMACRO{\dsum \limits_{j=0}^{v}}%
%BeginExpansion
{\displaystyle\sum\limits_{j=0}^{v}}
%EndExpansion
B_{j}^{v}(x)f_{\mathbb{B},k-j}(x,t). \label{A8}%
\end{equation}

\end{theorem}

\begin{proof}
Formula (\ref{A8}) follows immediately from (\ref{n3}).
\end{proof}

Using definition (\ref{1BB3}) and (\ref{1BB3a}) in Theorem \ref{TheoremRECUr},
we obtain a recurrence relation for the Bernstein basis functions:

\begin{theorem}
\label{TheoremRECUrTT}%
\begin{equation}
B_{k}^{n}(x)=%
%TCIMACRO{\dsum \limits_{j=0}^{v}}%
%BeginExpansion
{\displaystyle\sum\limits_{j=0}^{v}}
%EndExpansion
B_{j}^{v}(x)B_{k-j}^{n-v}(x). \label{n5}%
\end{equation}

\end{theorem}

\begin{proof}
By substituting the right hand side of (\ref{A-0}) into (\ref{A8}), we get%
\[
\frac{\partial^{v}}{\partial t^{v}}\left(  \sum_{n=0}^{\infty}B_{k}%
^{n}(x)\frac{t^{n}}{n!}\right)  =\sum_{n=0}^{\infty}\left(
%TCIMACRO{\dsum \limits_{j=0}^{v}}%
%BeginExpansion
{\displaystyle\sum\limits_{j=0}^{v}}
%EndExpansion
B_{j}^{v}(x)B_{k-j}^{n}(x)\right)  \frac{t^{n}}{n!}.
\]
Therefore%
\[
\sum_{n=v}^{\infty}B_{k}^{n}(x)\frac{t^{n-v}}{\left(  n-v\right)  !}%
=\sum_{n=0}^{\infty}\left(
%TCIMACRO{\dsum \limits_{j=0}^{v}}%
%BeginExpansion
{\displaystyle\sum\limits_{j=0}^{v}}
%EndExpansion
B_{j}^{v}(x)B_{k-j}^{n}(x)\right)  \frac{t^{n}}{n!}.
\]
From the above equation, we get%
\[
\sum_{n=v}^{\infty}B_{k}^{n}(x)\frac{t^{n-v}}{\left(  n-v\right)  !}%
=\sum_{n=v}^{\infty}\left(
%TCIMACRO{\dsum \limits_{j=0}^{v}}%
%BeginExpansion
{\displaystyle\sum\limits_{j=0}^{v}}
%EndExpansion
B_{j}^{v}(x)B_{k-j}^{n-v}(x)\right)  \frac{t^{n-v}}{\left(  n-v\right)  !}.
\]
Comparing the coefficients of $\frac{t^{n}}{n!}$ on the both sides of the
above equation, we arrive at the desired result.
\end{proof}

\begin{remark}
Setting $v=1$ in (\ref{n5}), one obtains the standard recurrence%
\[
B_{k}^{n}(x)=(1-x)B_{k}^{n-1}(x)+xB_{k-1}^{n-1}(x).
\]

\end{remark}

\subsection{Degree raising}

In this section we present a functional equation which we apply to provide a
new proof of the degree raising formula for the Bernstein polynomials.

From (\ref{1BB3}), we obtain the following \textbf{functional equation}%
\textit{:}%
\[
\left(  xt\right)  ^{d}f_{\mathbb{B},k}(x,t)=\frac{(k+d)!}{k!}f_{\mathbb{B}%
,k+d}(x,t).
\]
Therefore%
\begin{equation}
x^{d}B_{k}^{n}(x)=\frac{n!(k+d)!}{k!(n+d)!}B_{k+d}^{n+d}(x). \label{1BB4}%
\end{equation}
Substituting $d=1$ into the above equation, we have%
\begin{equation}
xB_{k}^{n}(x)=\frac{k+1}{n+1}B_{k+1}^{n+1}(x). \label{1BB4a}%
\end{equation}
The above relation can also be proved by (\ref{1BB3a}) cf. (\cite{GoldmanBOOK}%
, \cite{GoldmanBook2}, \cite{GoldmanBook3}).

From (\ref{1BB3}), we also get the following functional equation\textit{:}%
\[
\left(  xt\right)  ^{-d}f_{\mathbb{B},k}(x,t)=\frac{(k-d)!}{k!}f_{\mathbb{B}%
,k-d}(x,t).
\]
Therefore%
\[
(1-x)^{d}B_{k}^{n}(x)=\frac{n!(n+d-k)!}{\left(  n+d\right)  !(n-k)!}%
B_{k}^{n+d}(x).
\]
Substituting $d=1$, we have%
\begin{equation}
(1-x)B_{k}^{n}(x)=\frac{(n+1-k)}{\left(  n+1\right)  }B_{k}^{n+1}(x).
\label{n7}%
\end{equation}
Adding (\ref{1BB4a}) and (\ref{n7}), we get the standard degree elevation
formula%
\[
B_{k}^{n}(x)=\frac{1}{n+1}\left(  \left(  k+1\right)  B_{k+1}^{n+1}(x)+\left(
n+1-k\right)  B_{k}^{n+1}(x)\right)  .
\]

\section{New identities}

In this section, using alternative forms of the generating functions,
functional equations and Laplace transform, we give many new identities for
the Bernstein basis functions.

Using (\ref{1BB3}), we obtain the following \textit{functional equations:}%
\begin{equation}
f_{\mathbb{B},k_{1}}(x,t)f_{\mathbb{B},k_{2}}(x,t)=\left(
\begin{array}
[c]{c}%
k_{1}+k_{2}\\
k_{1}%
\end{array}
\right)  \frac{1}{2^{k_{1}+k_{2}}}f_{\mathbb{B},k_{1}+k_{2}}(x,2t), \label{A9}%
\end{equation}
and%
\begin{equation}
f_{\mathbb{B},k}(x,t)f_{\mathbb{B},k}(y,-t)=\frac{\left(  -xyt^{2}\right)
^{k}}{\left(  k!\right)  ^{2}}e^{t\left(  y-x\right)  }. \label{A9a}%
\end{equation}

\begin{theorem}%
\[
B_{k_{1}+k_{2}}^{n}(x)=\frac{2^{k_{1}+k_{2}-n}k_{1}!k_{2}!}{\left(
k_{1}+k_{2}\right)  !}%
%TCIMACRO{\dsum \limits_{j=0}^{n}}%
%BeginExpansion
{\displaystyle\sum\limits_{j=0}^{n}}
%EndExpansion
\left(
\begin{array}
[c]{c}%
n\\
j
\end{array}
\right)  B_{k_{1}}^{j}(x)B_{k_{2}}^{n-j}(x).
\]

\end{theorem}

\begin{proof}
By substituting the right hand side of (\ref{A-0}) into (\ref{A9}), we get%
\[
\sum_{n=0}^{\infty}B_{k_{1}}^{n}(x)\frac{t^{n}}{n!}\sum_{n=0}^{\infty}%
B_{k_{2}}^{n}(x)\frac{t^{n}}{n!}=\sum_{n=0}^{\infty}B_{k_{1}+k_{2}}%
^{n}(x)\frac{2^{n-k_{1}-k_{2}}\left(  k_{1}+k_{2}\right)  !t^{n}}%
{n!k_{1}!k_{2}!}.
\]
Therefore%
\[
\sum_{n=0}^{\infty}\left(
%TCIMACRO{\dsum \limits_{j=0}^{n}}%
%BeginExpansion
{\displaystyle\sum\limits_{j=0}^{n}}
%EndExpansion
\left(
\begin{array}
[c]{c}%
n\\
j
\end{array}
\right)  B_{k_{1}}^{j}(x)B_{k_{2}}^{n-j}(x)\right)  \frac{t^{n}}{n!}%
=\sum_{n=0}^{\infty}B_{k_{1}+k_{2}}^{n}(x)\frac{2^{n-k_{1}-k_{2}}\left(
k_{1}+k_{2}\right)  !t^{n}}{n!k_{1}!k_{2}!}.
\]
Comparing the coefficients of $\frac{t^{n}}{n!}$ on the both sides of the
above equation, we arrive at the desired result.
\end{proof}

\begin{theorem}%
\[
\left(  -xy\right)  ^{k}\left(  y-x\right)  ^{n-2k}=\frac{\left(  k!\right)
^{2}}{(n)_{2k}}%
%TCIMACRO{\dsum \limits_{j=0}^{n}}%
%BeginExpansion
{\displaystyle\sum\limits_{j=0}^{n}}
%EndExpansion
\left(
\begin{array}
[c]{c}%
n\\
j
\end{array}
\right)  (-1)^{n-j}B_{k}^{j}(x)B_{k}^{n-j}(y)
\]
where
\[
\left(  n\right)  _{2k}=n\left(  n-1\right)  ...\left(  n-2k+1\right)  ,
\]
and $\left(  n\right)  _{0}=1$.
\end{theorem}

\begin{proof}
Combining (\ref{A-0}) and (\ref{A9a}), we get%
\[
\left(  \sum_{n=0}^{\infty}B_{k}^{n}(x)\frac{t^{n}}{n!}\right)  \left(
\sum_{n=0}^{\infty}(-1)^{n}B_{k}^{n}(y)\frac{t^{n}}{n!}\right)  =\frac{\left(
-xy\right)  ^{k}}{\left(  k!\right)  ^{2}}\sum_{n=0}^{\infty}\frac{\left(
y-x\right)  ^{n}t^{n+2k}}{n!}.
\]
From the above equation, we obtain%
\[
\sum_{n=0}^{\infty}\left(
%TCIMACRO{\dsum \limits_{j=0}^{n}}%
%BeginExpansion
{\displaystyle\sum\limits_{j=0}^{n}}
%EndExpansion
\left(
\begin{array}
[c]{c}%
n\\
j
\end{array}
\right)  (-1)^{n-j}B_{k}^{j}(x)B_{k}^{n-j}(y)\right)  \frac{t^{n}}{n!}%
=\frac{\left(  -xy\right)  ^{k}}{\left(  k!\right)  ^{2}}\sum_{n=0}^{\infty
}\left(  n\right)  _{2k}\left(  y-x\right)  ^{n-2k}\frac{t^{n}}{n!}.
\]
Comparing the coefficients of $\frac{t^{n}}{n!}$ on the both sides of the
above equation, we arrive at the desired result.
\end{proof}

\begin{theorem}
\label{Tg1}Let $x\neq0$. For all positive integers $k$ and $n$, we have%
\[%
%TCIMACRO{\dsum \limits_{j=0}^{n-k}}%
%BeginExpansion
{\displaystyle\sum\limits_{j=0}^{n-k}}
%EndExpansion
\left(
\begin{array}
[c]{c}%
n\\
j
\end{array}
\right)  x^{j-k}B_{k}^{n-j}(x)=\left(
\begin{array}
[c]{c}%
n\\
k
\end{array}
\right)  .
\]

\end{theorem}

\begin{proof}
By using (\ref{g1}), we obtain%
\[
\sum_{n=0}^{\infty}B_{k}^{n}(x)\frac{t^{n}}{n!}\sum_{n=0}^{\infty}x^{n}%
\frac{t^{n}}{n!}=\frac{t^{k}x^{k}}{k!}\sum_{n=0}^{\infty}\frac{t^{n}}{n!}.
\]
Therefore%
\[
\sum_{n=k}^{\infty}\left(  \sum_{j=0}^{n-k}\left(
\begin{array}
[c]{c}%
k\\
j
\end{array}
\right)  x^{j}B_{k}^{n-j}(x)\right)  \frac{t^{n}}{n!}=x^{k}\sum_{n=k}^{\infty
}\left(
\begin{array}
[c]{c}%
n\\
k
\end{array}
\right)  \frac{t^{n}}{n!}.
\]
Comparing the coefficients of $\frac{t^{n}}{n!}$ on the both sides of the
above equation, we arrive at the desired result.
\end{proof}

\begin{theorem}
\label{Tg2}For all positive integers $k$ and $n$, we have%
\[%
%TCIMACRO{\dsum \limits_{j=0}^{n-k}}%
%BeginExpansion
{\displaystyle\sum\limits_{j=0}^{n-k}}
%EndExpansion
(-1)^{j}\left(
\begin{array}
[c]{c}%
n\\
j
\end{array}
\right)  B_{k}^{n-j}(x)=(-1)^{n-k}\left(
\begin{array}
[c]{c}%
n\\
k
\end{array}
\right)  x^{n}.
\]

\end{theorem}

\begin{proof}
By using (\ref{g2}), we get%
\[
\sum_{n=0}^{\infty}B_{k}^{n}(x)\frac{t^{n}}{n!}\sum_{n=0}^{\infty}%
(-1)^{n}\frac{t^{n}}{n!}=\frac{t^{k}x^{k}}{k!}\sum_{n=0}^{\infty}(-1)^{n}%
x^{n}\frac{t^{n}}{n!}.
\]
Therefore%
\[
\sum_{n=k}^{\infty}\left(  \sum_{j=0}^{n-k}(-1)^{j}\left(
\begin{array}
[c]{c}%
k\\
j
\end{array}
\right)  B_{k}^{n-j}(x)\right)  \frac{t^{n}}{n!}=\sum_{n=k}^{\infty}\left(
\begin{array}
[c]{c}%
n\\
k
\end{array}
\right)  (-1)^{n-k}x^{n}\frac{t^{n}}{n!}.
\]
Comparing the coefficients of $\frac{t^{n}}{n!}$ on the both sides of the
above equation, we arrive at the desired result.
\end{proof}

\begin{theorem}
\label{Tg5}For all positive integers $k$ and $n$, we have%
\[%
%TCIMACRO{\dsum \limits_{j=0}^{n-k}}%
%BeginExpansion
{\displaystyle\sum\limits_{j=0}^{n-k}}
%EndExpansion
(-1)^{j}\left(
\begin{array}
[c]{c}%
n\\
j
\end{array}
\right)  \left(  1-x\right)  ^{j}B_{k}^{n-j}(x)=\left\{
\begin{array}
[c]{c}%
x^{k}\text{, for }x=k,\\
\\
0\text{, for }x\neq k.
\end{array}
\right.
\]

\end{theorem}

\begin{proof}
Proof of Theorem \ref{Tg5} is same as that of Theorem \ref{Tg1}. So we omit it.
\end{proof}

\section{Applications of the Laplace transform to the generating functions for
the Bernstein basis functions}

In this section, we give some applications of the Laplace transform to the
generating functions for the Bernstein basis functions. We obtain interesting
series representations for the Bernstein basis functions.

\begin{theorem}
\label{Tg3}Let $x\neq0$. For all positive integer $k$, we have%
\[
\sum_{n=0}^{\infty}xB_{k}^{n}(x)=1.
\]

\end{theorem}

\begin{proof}
Integrate equation (\ref{g2}) (by parts) with respect to $t$ from $0$ to
$\infty$, we have%
\begin{equation}
\sum_{n=0}^{\infty}B_{k}^{n}(x)\frac{1}{n!}%
%TCIMACRO{\dint \limits_{0}^{\infty}}%
%BeginExpansion
{\displaystyle\int\limits_{0}^{\infty}}
%EndExpansion
t^{n}e^{-t}dt=\frac{x^{k}}{k!}%
%TCIMACRO{\dint \limits_{0}^{\infty}}%
%BeginExpansion
{\displaystyle\int\limits_{0}^{\infty}}
%EndExpansion
t^{k}e^{-xt}dt. \label{gg1}%
\end{equation}
If we appropriately use the case%
\[
x>0
\]
of the following Laplace transform of the function $f(t)=t^{k}$:%
\[%
%TCIMACRO{\tciLaplace}%
%BeginExpansion
\mathcal{L}%
%EndExpansion
(t^{k})=\frac{k!}{x^{k+1}},
\]
on the both sides of (\ref{gg1}), we find that%
\[
\sum_{n=0}^{\infty}B_{k}^{n}(x)=\frac{1}{x}.
\]
From the above equation, we arrive at the desired result.
\end{proof}

\begin{theorem}
\label{Tg4}Let $x\neq0$. For all positive integer $k$, we have%
\[
\sum_{n=0}^{\infty}(-1)^{n}\frac{B_{k}^{n}(x)}{x^{n+1}}=(-1)^{k}x^{k}.
\]

\end{theorem}

\begin{proof}
Proof of Theorem \ref{Tg4} is same as that of Theorem \ref{Tg3}. That is if
we\ replace $t$ by $-t$\ in equation (\ref{g1}) and integrate by parts with
respect to $t$ from $0$ to $\infty$ and using Laplace transform of the
function $f(t)=t^{n}$, then we arrive at the desired result.
\end{proof}

\section{Further Remarks and Observations}

Fourier series of the Bernstein polynomials has been studied, without
generating functions, by Izumi \textit{et al}. \cite{izumi}. They investigated
many properties of the Fejer mean of the Fourier series of these polynomials.
Fourier transform of the Bernstein polynomials has also been given, without
generating functions, by Chui \textit{at al}. \cite{Chui}. By replacing $t$ by
$it$ in (\ref{g1})-(\ref{g3}), one may give applications of the Fourier
transform to the \textit{complex} generating functions for the Bernstein basis functions.

\begin{acknowledgement}
The author would like to thank Professor Ronald Goldman (Rice University,
Houston, USA) for his very valuable comments and his suggestions.

The present investigation was supported by the \textit{Scientific Research
Project Administration of Akdeniz University.}
\end{acknowledgement}

\end{document}